\documentclass[11pt,a4paper,oneside, reqno]{amsart}
\usepackage{mathrsfs}
\usepackage{amsfonts}
\usepackage{amssymb}
\usepackage{amsxtra}
\usepackage{dsfont}
\usepackage{amsthm}
\usepackage{geometry}
\usepackage{float}
\usepackage{amsmath, amssymb, amsthm}
\usepackage[utf8]{inputenc}
\usepackage[T1]{fontenc}
\usepackage{graphicx}
\usepackage[unicode, psdextra]{hyperref}
\usepackage{listings}
\usepackage{indentfirst}
\usepackage{theoremref}
\usepackage{thmtools}
\usepackage{physics}
\usepackage{mathtools}
\usepackage[english,polish]{babel}
\usepackage{color}

\newcommand{\pl}[1]{\foreignlanguage{polish}{#1}}

\usepackage{hyphenat}

\newcommand{\innprod}[2]{\left< #1, #2 \right>}
\newcommand{\R}{\mathbb{R}}
\newcommand{\RR}{\mathbb{R}}
\renewcommand{\C}{\mathbb{C}}

\newcommand{\Z}{\mathbb{Z}}

\definecolor{green}{rgb}{0,0.5,0}

\newtheorem{theorem}{Theorem}[section]
\newtheorem{lemma}[theorem]{Lemma}
\newtheorem{corollary}[theorem]{Corollary}
\theoremstyle{remark}
\newtheorem{remark}{Remark}

\theoremstyle{definition}

\numberwithin{equation}{section}

\allowdisplaybreaks

\title[A dimension-free estimate on $L^2$ for the maximal Riesz transform]{A dimension-free estimate on $L^2$ for the maximal Riesz transform in terms of the Riesz transform}

\author{Maciej Kucharski}
\address{Maciej Kucharski\\
	Instytut Matematyczny\\
	Uniwersytet \pl{Wroc{\lll}awski}\\
	Plac Grun\-waldzki 2/4\\
	50-384 \pl{Wroc{\lll}aw}\\
	Poland}
\email{mkuchar@math.uni.wroc.pl}

\author{B{\l}a{\.z}ej Wr{\'o}bel}
\address{ B{\l}a{\.z}ej Wr{\'o}bel\\
	Instytut Matematyczny\\
	Uniwersytet \pl{Wroc{\lll}awski}\\
	Plac Grun\-waldzki 2/4\\
	50-384 \pl{Wroc{\lll}aw}\\
	Poland}
\email{blazej.wrobel@math.uni.wroc.pl}


\thanks{Both authors were supported by the National Science Centre (NCN), Poland  research project Preludium Bis 2019/35/O/ST1/00083.}

\begin{document}
	\selectlanguage{english}
	\begin{abstract}
	W prove a dimension-free estimate for the $L^2(\R^d)$ norm of the  maximal truncated Riesz transform in terms of the $L^2(\R^d)$ norm of the Riesz transform. Consequently, the vector of maximal truncated Riesz transforms has a dimension-free estimate on $L^2(\R^d).$  We also show that the maximal function of the vector of truncated Riesz transforms has a dimension-free estimate on all $L^p(\R^d)$ spaces, $1<p<\infty.$
	\end{abstract}
	
	\maketitle

	\section{Introduction}
\label{sec:int}

The Riesz transforms $R_j,$ $j=1,\ldots,d,$ are an archetype of singular integral operators on $\R^d$. For a Schwartz function $f$ they are defined by
\begin{equation} \label{eq:riesz}
	R_j f(x) = \lim_{t \to 0^+} R_j^t f(x), 
\end{equation}
where
\begin{equation} \label{eq:Rt}
	R_j^t f(x) = \frac{\Gamma\left( \frac{d+1}{2} \right)}{\pi^{\frac{d+1}{2}}} \int_{\abs{x-y} > t} \frac{x_j - y_j}{\abs{x-y}^{d+1}} f(y) \, dy
\end{equation}
is the truncated Riesz transform. The proof of the boundedness of the Riesz transforms on $L^p(\R^d)$ spaces for $1<p<\infty$ is a prime example of the Calder\'on--Zygmund theory of singular integral operators initiated in the seminal paper \cite{CZ1}. An important ingredient of the proof is the $L^2(\R^d)$ boundedness of $R_j$, which follows directly from  Plancherel's theorem and the well-known formula
\begin{equation}
	\label{eq:ftRiesz}
	\widehat{R_j f}(\xi)=-i\frac{\xi_j}{|\xi|} \widehat{f}(\xi),\qquad \xi \in \R^d;
\end{equation}
where $\hat{g}$ denotes the (unitary) Fourier transform of $g\in L^2(\R^d).$

An application of the Calder\'on--Zygmund theory shows that the $L^p(\RR^d)$ norm of $R_j$ grows at most exponentially in the dimension $d.$ In \cite{stein1} E. M. Stein proved that the vector of Riesz transforms
$(R_1f,\ldots,R_df)$
has $L^p(\R^d)$  bounds which are independent of the dimension. More precisely,
\begin{equation}
	\label{eq:Riesz0}
	\norm{\bigg(\sum_{j=1}^d |R_jf|^2\bigg)^{1/2}}_{L^p(\R^d)}\leqslant C_p\, \|f\|_{L^p(\R^d)},\qquad 1<p<\infty,\end{equation}
 where $C_p$ is independent of the dimension $d.$ Note that in view of \eqref{eq:ftRiesz} the inequality \eqref{eq:Riesz0} is clear on $L^2(\R^d)$ with $C_2=1.$ Later it was realized that for $1<p<2$ one may take $C_p\leqslant C (p-1)^{-1}$ in \eqref{eq:Riesz0}, where $C$ is a universal constant, see \cite{Ba1}, \cite{rubio}. Still the optimal constant in \eqref{eq:Riesz0} remains unknown when $d\geqslant 2;$ the best results to date are given in \cite{BW1} (see also \cite{DV}) and \cite{IM1}.

The maximal truncated Riesz transform is the maximal function counterpart of the singular integral \eqref{eq:riesz} defining $R_j.$ Namely, for $j=1,\ldots,d$ we set
\begin{equation*} 
	R_j^* f(x) = \sup_{t>0} |R_j^t f(x)|. 
\end{equation*}
Clearly, for $f$ being a Schwartz function and $x\in \R^d$ we have the pointwise inequality $R_j^*f(x)\geqslant R_j f(x).$ In \cite[Theorem 1]{mateu_verdera} Mateu and Verdera proved that, up to a multiplicative constant, also the reverse inequality holds in the $L^p(\R^d)$ norm. Namely, they showed that there exists a constant $C_{p,d}$ such that
\begin{equation} \label{eq:mat_ver}
	\norm{R_j^* f}_{L^p(\R^d)}\leqslant C_{p,d} \norm{R_j f}_{L^p(\R^d)}
\end{equation}
for all $f\in L^p(\R^d)$ with $1<p<\infty.$ We remark that in dimensions $d\geqslant 2$ inequality \eqref{eq:mat_ver} does not follow from classical Cotlar's inequality applied to the Riesz transform
	\begin{equation*}
		R_j^* f\leqslant \mathcal M (R_j f)+ C_d\, \mathcal M(f),
	\end{equation*}
see e.g.\ \cite[Theorem 5.3.4]{grafakos}. This is because of the presence of the second term $\mathcal M(f)$ which contains only the Hardy--Littlewood maximal function $\mathcal M$ and not the Riesz transform $R_j.$

The proof of \eqref{eq:mat_ver} given in \cite{mateu_verdera} produces a constant $C_{p,d}$ of an exponential order $C_p\cdot 4^d,$ where $C_p$ depends only on $p.$ In the present paper we prove that for $p=2$ one may take a universal constant.

\begin{theorem} \thlabel{thm}
For every $f \in L^2(\R^d)$ we have
	\begin{equation}
		\label{eq:thm}
		\norm{R_j^* f}_{L^2(\R^d)} \leqslant 2\cdot 10^8 \norm{R_j f}_{L^2(\R^d)}.
	\end{equation}
\end{theorem}
\begin{remark}
	There is nothing special in $2\cdot 10^8 $ and clearly, optimizing our method, one may get a better constant instead. We wrote down an explicit constant in order to get an impression of its magnitude.
\end{remark}

\begin{remark}
	The theorem is true for all dimensions $d$, however we restrict our proof to the case $d \geqslant 4$ due to technical reasons and from now on we assume that $d \geqslant 4$. The case $1 \leqslant d \leqslant 3$ follows from \cite[Theorem 1]{mateu_verdera}.
\end{remark}

Theorem \ref{thm} is the main result of our paper. It is an example of a control of a singular integral by its maximal truncated singular integral. This theme has been developed in \cite{bmo1, mopv1, mov1} for broad classes of singular integral operators and we kindly refer the reader to these papers for more details. We remark that for stating a result with a dimension-free control of norms one needs a definition of a family of operators in each dimension $d.$ For this purpose the Riesz transforms seem the most natural candidates. We do not know if a dimension-free control as in \eqref{eq:thm} is feasible between other singular integral operators and their maximal counterparts.

Note that Theorem \ref{thm} combined with Plancherel's theorem and \eqref{eq:ftRiesz} easily implies a dimension-free bound for the norm of the vector of maximal Riesz transforms on $L^2(\RR^d).$ 
\begin{corollary}
	\label{cor:vecmaxL2}
	For every $f \in L^2(\R^d)$ we have
	\[
		\norm{\bigg( \sum_{j=1}^d \abs{R_j^* f}^2 \bigg)^{1/2}}_{L^2(\R^d)} \leqslant 2\cdot 10^8 \norm{f}_{L^2(\R^d)}.
	\]
\end{corollary}

In the course of investigating the relation between the Riesz transforms and the maximal truncated Riesz transforms we also dealt with the maximal function associated with the vector of truncated Riesz transforms, i.e. the operator defined as
\begin{equation*}
\sup_{t > 0} \abs{R^t f} \coloneqq \sup_{t > 0} \bigg( \sum_{j=1}^d \abs{R_j^t f}^2 \bigg)^{1/2}.
\end{equation*}
Note that the difference between $\left( \sum_{j=1}^d \abs{R_j^* f}^2 \right)^{1/2}$ and $\sup_{t > 0} \abs{R^t f}$ lies in the placement of the supremum. It turns out that the latter operator   is significantly easier to estimate. In particular we have a dimension-free bound for its $L^p(\RR^d)$ norm for all $p \in (1, \infty)$.

\begin{theorem}
	\thlabel{thm2}
For each $p \in (1, \infty)$ there exists a constant $C_p > 0$ independent of the dimension $d$ and such that
	\[
	\norm{\sup_{t > 0} \left( \sum_{j=1}^d \abs{R_j^t f}^2 \right)^{1/2}}_{L^p(\RR^d)} \leqslant C_p \norm{f}_{L^p(\RR^d)}
	\]
	for all $f \in L^p(\R^d)$.
\end{theorem}

We shall now describe the methods employed in our paper, starting with \thref{thm}. There are two main ingredients used in its proof. Firstly, we need a factorization of the truncated Riesz transform $R_j^t=M^t(R_j)$, where $M^t,$ $t>0,$ is a family of multiplier operators. This is stated in Corollary \ref{cor:fact} and is a crucial part of our argument. Such a factorization does not seem to appear in the literature. However, it is related to \cite[Section 2]{mateu_verdera}, see Remark \ref{rem:factker} for details. An application of Corollary \ref{cor:fact} reduces \thref{thm} to estimating the $L^2(\RR^d)$ norm of the maximal operator corresponding to the operators $M^t,$ see  \thref{thm'}. The proof of this theorem is based on methods initiated by Bourgain in connection with dimension-free estimates for Hardy--Littlewood averaging operators over convex sets in \cite{bou1} and their later developments, see e.g.\ \cite{bourgain_wrobel}, \cite{mirek_stein}. In order to draw a dimension-free statement from these methods we need to justify appropriate dimension-free estimates for the multiplier symbols corresponding to $M^t,$ $t>0.$ This is accomplished in \thref{lem2,lem3,lem4}.

In order to prove  \thref{thm2} we use appropriately the method of rotations. This is a small variation on the paper \cite{rubio} by Duoandikoetxea and Rubio de Francia.

\section{Notation and Preliminaries} \label{sec:prel}
We abbreviate $L^p(\R^d)$ to $L^p$ and $\norm{\cdot}_{L^p}$ to $\norm{\cdot}_p$. For a sublinear operator  $T$ on $L^p$ we denote by $\|T\|_{p\to p}$ its norm. We  let $\mathcal S$ be the space of Schwartz functions on $\R^d.$ Slightly abusing the notation we say that a sublinear operator $T$  is bounded on $L^p$ if it is bounded on $\mathcal S$ in the $L^p$ norm. This is convenient since all the operators (linear and sublinear) considered in this paper are naturally defined on Schwartz functions. We often do not distinguish between a bounded operator $T$ on $\mathcal S \cap L^p$ and its unique bounded extension to $L^p.$ Moreover, if no explicit condition is imposed on a function $f$, we tacitly assume that $f\in \mathcal S.$   


We will use some inequalities involving the gamma function. The first one is a sharper variant of Stirling's formula \cite[5.6.1]{nist}
\begin{equation} \label{eq:stirling}
	\sqrt{2\pi} x^{x-\frac{1}{2}} e^{-x} \leqslant \Gamma(x) \leqslant \sqrt{2\pi} x^{x-\frac{1}{2}} e^{-x + \frac{1}{12x}}, \qquad x > 0.
\end{equation}
The second one is Gautschi's inequality \cite[5.6.4]{nist}
\begin{equation} \label{eq:gautschi}
	x^{1-s} < \frac{\Gamma(x+1)}{\Gamma(x+s)} < (x+1)^{1-s}, \qquad x > 0, s \in (0,1).
\end{equation}

In the proof of \thref{thm'} we shall need a numerical inequality (see e.g.\ \cite[Lemma 2.5]{mirek_stein}) which says that for any $n \in \Z$ and continuous function $g: [2^n, 2^{n+1}] \to \C$ we have
\begin{equation} \label{eq:ineq}
	\begin{aligned}
		\sup_{t \in [2^n, 2^{n+1}]} &\abs{g(t) - g(2^n)} \\
		&\leqslant \sqrt{2} \sum_{l = 0}^\infty \left( \sum_{m=0}^{2^l-1} \abs{g(2^n + 2^{n-l}(m+1)) - g(2^n + 2^{n-l}m)}^2 \right)^{1/2}.
	\end{aligned}
\end{equation}

Throughout the paper $c_d$ denotes the constant
\begin{equation}
	\label{eq:cd}
	c_d = \frac{\Gamma\left( \frac{d+1}{2} \right)}{\pi^{\frac{d+1}{2}}}
\end{equation}
and $K_j^t$  denotes the kernel
\[
	K_j^t(x) = c_d \chi_{\abs{x} > t}(x) \frac{x_j}{\abs{x}^{d+1}},\qquad x\in \R^d,
\]
where $j=1,\ldots,d$ and $t>0.$
Then $R_j^t f = K_j^t * f$ so that $K_j^t$ is the kernel of the $j$-th truncated Riesz transform defined in \eqref{eq:Rt}.

The Fourier transform is defined for $\xi\in\R^d$  by the formula
\[
\widehat{f}(\xi) = \int_{\R^d} f(x) e^{-2 \pi i x \cdot \xi} dx.
\]

Let $P_t$ be the Poisson semigroup  defined for $f\in L^2$ by
\begin{equation} \label{eq:poisson}
	\widehat{P_t f}(\xi) = p_t(\xi) \widehat{f}(\xi) \qquad \textrm{ with }\quad p_t(\xi) = e^{-t \frac{\abs{\xi}}{\sqrt{d}}}.
\end{equation}
We denote by $P_*(f)$ and $g(f)$ the maximal function and the square function associated with this semigroup, i.e.
\begin{equation*} 
	P_*f(x) = \sup_{t > 0} \abs{P_t f(x)}  \quad \text{and} \quad g(f)(x) = \left( \int_0^\infty t \abs{\frac{d}{dt} P_t f(x)}^2 dt \right)^{1/2}.
\end{equation*}
From \cite[pp. 47--51]{stein2} and \cite[Theorem VIII.7.7]{dunford_schwartz} we know that for  $f \in L^2$ we have
\begin{equation} \label{eq:P*g}
	\norm{P_* f}_2 \leqslant 4 \norm{f}_2 \quad \text{and} \quad \norm{g(f)}_2 \leqslant \frac{1}{\sqrt{2}} \norm{f}_2.
\end{equation}
We will also need the so-called Poisson projections given by
\[
	S_n = P_{2^{n-1}} - P_{2^n}, \quad n \in \Z.
\]
The sequence $(S_n)_{n \in \Z}$ is then a resolution of the identity on $L^2$ which means that
\begin{equation} \label{eq:Sn}
	f = \sum_{n \in \Z} S_n f, \quad f \in L^2.
\end{equation}
Moreover, $S_n$ satisfies
\[
	S_n f(x) = -\int_{2^{n-1}}^{2^n} \frac{d}{dt} P_tf(x) dt,
\]
so by the Cauchy--Schwartz inequality we have
\begin{align*}
	\abs{S_n f(x)}^2 \leqslant 2^{n-1} \int_{2^{n-1}}^{2^n} \abs{\frac{d}{dt} P_tf(x)}^2 dt \leqslant \int_{2^{n-1}}^{2^n} t \abs{\frac{d}{dt} P_tf(x)}^2 dt.
\end{align*}
Now summing over $n \in \Z$ and using \eqref{eq:P*g} leads us to the conclusion that
\begin{equation} \label{eq:Snnorm}
	\norm{\left( \sum_{n \in \Z} \abs{S_n f}^2 \right)^{1/2}}_2 \leqslant \frac{1}{\sqrt{2}} \norm{f}_2.
\end{equation}

Before we proceed let us mention that there is another way of defining the Riesz transform which results in a different maximal transform. However, as we will see, this case is easier to deal with. Let 
\[
P_t(x) = \frac{c_d}{2\pi\sqrt{d}} \frac{t}{\left( \left( \frac{t}{2\pi\sqrt{d}} \right)^2 + \abs{x}^2 \right)^\frac{d+1}{2}} \qquad \text{and} \qquad Q_t^j(x) = c_d \frac{x_j}{\left( \left( \frac{t}{2\pi\sqrt{d}} \right)^2 + \abs{x}^2 \right)^\frac{d+1}{2}}
\]
be the {Poisson kernel} and the {conjugate Poisson kernel}, recall that $c_d$ was defined in \eqref{eq:cd}. Note that we rescale the usual definition of the Poisson kernel to match \eqref{eq:poisson}. Their Fourier transform are
\[
\widehat{P_t}(\xi) = e^{-t \frac{\abs{\xi}}{\sqrt{d}}} \qquad \text{and} \qquad \widehat{Q_t^j}(\xi) = -i \frac{\xi_j}{\abs{\xi}} e^{-t \frac{\abs{\xi}}{\sqrt{d}}},
\]
which, together with the fact that the family $(P_t)_{t > 0}$ is an approximate identity, immediately implies that for $f\in \mathcal S$ we have  $Q_t^j * f = R_j f * P_t \xrightarrow{t \to 0^+} R_j f$ pointwise. Thus, an equivalent {definition} of the Riesz transform is $R_j f(x) = \lim_{t \to 0^+} Q_t^j * f(x)$. This leads to the maximal Riesz transform
\[
R_j^{*,P} f(x) = \sup_{t > 0} \abs{Q_t^j * f(x)}.
\]

Now observe that the kernels $(P_t)_{t > 0}$ are $L^1$ dilations of an integrable radially decreasing function $P_1$ of $L^1$ norm equal to $1$, so by \cite[Theorem 2.1.10]{grafakos} we have
\begin{equation}
\label{eq:Rj*P}
R_j^{*,P} f = \sup_{t > 0} \abs{Q_t^j * f} = \sup_{t > 0} \abs{R_j f * P_t} \leqslant \norm{P_1}_1\cdot  \mathcal{M}(R_jf)=\mathcal{M}(R_jf),
\end{equation}
where $\mathcal{M}$ is the centered Hardy--Littlewood maximal operator. In view of a dimension-free estimate of $\mathcal{M}$ due to Stein \cite[Theorem 13]{SteinMax} (see also \cite{SteinStro}) this means that the inequality
\[
\norm{R_j^{*,P} f}_p \leqslant C_p \norm{R_j f}_p
\]
holds for $p \in (1, \infty]$ with the constant $C_p$ independent of the dimension $d$. Moreover, a recent result of Deleaval and Kriegler \cite[Theorem 1]{deleaval} provides dimension-free bounds for the vector-valued Hardy--Littlewood maximal operator. Hence, using \cite[Theorem 1]{deleaval} together with \eqref{eq:Rj*P} and \eqref{eq:Riesz0} we obtain
\begin{equation}
	\label{eq:vecmaxLpPois}
\norm{\left( \sum_{j=1}^d \abs{R_j^{*,P} f}^2 \right)^{1/2}}_p \leqslant C_p \norm{f}_p
\end{equation}
for $p \in (1, \infty)$ with $C_p$ independent of the dimension $d$.

Proving a version of \eqref{eq:vecmaxLpPois} with $R_j^{*,P}$ replaced by $R_j^*$ seems a much harder task. 


\section{A factorization of the truncated Riesz transform and multiplier estimates} \label{sec:lemmas}
First we state and prove results regarding the multipliers associated with the truncated Riesz transforms. Let
$m: [0, +\infty) \to \C$ be the function
\begin{equation} \label{eq:m}
	m(x) = \frac{2^{\frac{d}{2}} \Gamma\left( \frac{d+1}{2} \right)}{\sqrt{\pi}} \int_{2\pi x}^\infty r^{-\frac{d}{2}} J_{\frac{d}{2}}(r) dr.
\end{equation}
For $\Re \nu > -\frac{1}{2}$ the symbol $J_\nu$ denotes the Bessel function of the first kind defined by
\begin{equation}
	\label{eq:besdef}
	J_\nu(t) = \frac{t^\nu}{2^\nu \Gamma\left( \nu + \frac{1}{2} \right) \sqrt{\pi}} \int_{-1}^1 e^{its} \left(1 - s^2 \right)^{\nu - \frac{1}{2}} ds, \qquad t \geqslant 0,
\end{equation}
see e.g.\ \cite[B.1]{grafakos}. It is known that for $\nu \geqslant 0$ the Bessel function satisfies ${\abs{J_\nu(t)} \leqslant 1}$ (see \cite[10.14.1]{nist}) and $\abs{J_\nu(t)} \leqslant C(\nu) t^{\nu}$ (see \cite[10.14.4]{nist}). Using the assumption that $d \geqslant 4$ we thus see that \eqref{eq:m} defines a bounded continuous function on $[0,\infty).$

\begin{lemma} \thlabel{lem1}
For each $t>0$ the multiplier associated with the $j$-th truncated Riesz transform defined in \eqref{eq:Rt} equals
	\[
		\widehat{K_j^t}(\xi) = -i \frac{\xi_j}{\abs{\xi}} m(t \abs{\xi}),
	\]
	where $\xi\in\R^d,$ $\xi\neq 0.$ 
\end{lemma}

\begin{proof}
	First observe that $K_j^t(x) = K_j^1 \left( \frac{x}{t} \right) t^{-d}$ which means that $\widehat{K_j^t}(\xi) = \widehat{K_j^1}(t \xi)$ and we can focus on $K_j \coloneqq K_j^1$. Then we write
	\[
		K_j(x) = c_d x_j \chi_{\abs{x}>1}(x) \frac{1}{\abs{x}^{d+1}} = x_j K(x) \qquad{\rm with}\qquad K(x) := c_d \chi_{\abs{x}>1}(x)\frac{1}{\abs{x}^{d+1}}
	\]
	so that $\widehat{K_j} = -\frac{1}{2\pi i} \partial_j \widehat{K}$. Since $K$ is radial, its Fourier transform $\widehat{K}$ is also radial and has the form $\widehat{K}(\xi) = h(\abs{\xi})$, where
	\begin{equation} \label{eq:h}
		h(x) = 2 \pi c_d\, x^{-\frac{d}{2}+1} \int_1^\infty r^{-\frac{d}{2}-1} J_{\frac{d}{2}-1} (2\pi r x) dr,
	\end{equation}
  see e.g.\ \cite[B.5]{grafakos}.
Recalling the estimate $\abs{J_{d/2-1}(x)} \leqslant 1$ we see that for $x>0$  the integral in \eqref{eq:h} is convergent and the function $h$ is well defined. Since by \cite[10.6.6]{nist} the Bessel function satisfies 
	\begin{equation}
		\label{eq:Beseq}
		\frac{1}{x} \frac{d}{dx} \frac{J_\alpha(x)}{x^\alpha} = -\frac{J_{\alpha+1}}{x^{\alpha+1}},
	\end{equation}
	differentiating \eqref{eq:h} for $x>0$ we obtain 
	\[
		h'(x) = -c_d (2\pi)^{\frac{d}{2}+1} \int_{2\pi x}^\infty r^{-\frac{d}{2}} J_{\frac{d}{2}}(r) dr.
	\]
	 Passing with the derivative under the integral sign in \eqref{eq:h}  can be easily justified with the aid of \eqref{eq:Beseq}. 
	In summary we have proved that
	\[
		\widehat{K_j}(\xi) = -\frac{1}{2 \pi i}\frac{\xi_j}{\abs{\xi}} h'(\abs{\xi}) = -i \frac{\xi_j}{\abs{\xi}} \left( -\frac{1}{2\pi} h'(\abs{\xi}) \right)
	\]
	and noticing that $-h'(|\xi|)=2\pi m(|\xi|)$ completes the reasoning.
\end{proof}

Let $M^t,$ $t>0,$ be defined by 
 \begin{equation}
 	\label{eq:Mtdef}
 \widehat{M^t f}(\xi) = m(t\abs{\xi})\widehat{f}(\xi),\qquad f\in L^2,
\end{equation}
 and set 
  \begin{equation}
  	\label{eq:Mtmaxdef}
 M^* f(x) = \sup_{t > 0} \abs{M^t f(x)}.
\end{equation} 
Since $m$ is a bounded function, Plancherel's theorem implies that $M^t$ defines a bounded operator on $L^2.$  Moreover, since $m$ is continuous, we see that if $f\in \mathcal S$, then for each $x\in \R^d$ the mapping $t\mapsto M^tf(x)$ is continuous. In particular for such $f$ the supremum in the definition of $M^* f(x)$ may be restricted to rational numbers, which shows that the function $M^*f(x)$ is Borel measurable.  

 As a Corollary of \thref{lem1} we shall obtain a factorization of $R_j^t$ in terms of $M^t$ that is crucial for our purposes, see Corollary \ref{cor:fact}. For its proof we need a lemma on the density of $R_j(L^p)$ in $L^p.$ For $p=2$ this is an easy consequence of Plancherel's theorem and \eqref{eq:ftRiesz}. The general case $1<p<\infty$ is justified below.
 
\begin{lemma} \thlabel{lem:dens}
	Let $1 < p < \infty$ and $j=1,\ldots,d.$ Then the space $R_j(L^p)\cap \mathcal S$ is dense in $L^p$. In particular $R_j(L^p)$ is dense in $L^p.$ 
\end{lemma}
\begin{proof}
	Throughout the proof we fix $1<p<\infty$ and $j=1,\ldots,d.$ It is sufficient to prove that $(R_j)^2(L^p)\cap \mathcal S$ is dense in $L^p$ and this is our goal. Here the symbol $(R_j)^2$ denotes the two-fold composition $(R_j)^2=R_j \circ R_j.$ 
	
	For $t>0$ and $f \in L^1+L^{\infty}$  denote
	\begin{equation*}
	T_t^j f(x)=(4\pi t)^{-1/2} \int_{\R} \exp (-|x_j-y_j|^2/4t) f(x_1,\ldots,x_{j-1},y_j,x_{j+1},\ldots,x_d)\,dy_j.
\end{equation*}
Then $\{T_t^j\}_{t>0}$ is the heat semigroup on $\R$ applied to the $j$-th coordinate of $\R^d.$ It is a symmetric diffusion semigroup in the sense of Stein \cite[Chapter 3]{stein2}. Applying the Fourier transform for $f\in \mathcal S$  we obtain
\begin{equation}
	\label{eq:FTtj}
	\widehat{T_t^j f}(\xi)=\exp(-4\pi^2 t\xi_j^2)\widehat{f}(\xi),\qquad \xi \in \R^d.
\end{equation} 
In particular if $f\in \mathcal S$, then also  $T_t^j f\in \mathcal S.$


Take  $f\in \mathcal S.$ It is easy to show  (see \cite[Proposition 5.1.17]{grafakos}) that
	\begin{equation}
		\label{eq:RijLap}
		(R_j)^2 (\Delta f)=-\partial_j^2 f,
\end{equation}
	where $\Delta$ denotes the Laplacian on $\R^d.$ Using the Fourier inversion formula together with \eqref{eq:FTtj} and \eqref{eq:RijLap} we obtain for each $t>0$
	\begin{align} \label{eq:lemdens}
		T_t^j f - f &= \int_0^t T_s^j (\partial_j^2  f) ds = -\int_0^t T_s^j ((R_j)^2(\Delta f)) ds = -(R_j)^2\left(\int_0^t T_s^j (\Delta f)\,ds\right).
	\end{align}
The integrals in \eqref{eq:lemdens} are Bochner integrals on $L^2.$ Since $f\in \mathcal S$, we see that $\lim_{t\to \infty}T_t^j f=0$ both pointwise and in the $L^2$ norm.    
	Now invoking the $L^p$ boundedness of the maximal operator $f \mapsto \sup_{t > 0} \abs{T_t^j f}$ (see \cite[Chapter III, Section 3]{stein2}) and the dominated convergence theorem we deduce that also ${\lim_{t \to \infty} \norm{T_t^j f}_p = 0}.$
	Thus, denoting $g_t= \int_0^t T_s^j (\Delta f)\,ds$ and coming back to \eqref{eq:lemdens} we see that $(R_j)^2(g_t)\in \mathcal S$ and 
	\[\lim_{t\to \infty}\norm{(R_j)^2(g_t)-f}_p=0.\]
	Noticing that $g_t \in \mathcal S$ we conclude that any $f\in \mathcal S$ may be approximated arbitrarily close in the $L^p$ norm by an element of $(R_j)^2(\mathcal S)\cap \mathcal S.$ At this point the density of $\mathcal S$ in $L^p$ completes the proof.
	
\end{proof}

Having proved \thref{lem:dens} we now have all the ingredients for justifying the factorization announced earlier. Recall that the operators $M^t$ and $M^*$ are defined by \eqref{eq:Mtdef} and \eqref{eq:Mtmaxdef}, respectively.

\begin{corollary}
	\label{cor:fact}
	Let $j=1,\ldots,d.$ Then for each $t>0$ the  truncated Riesz transform factorizes as
	\begin{equation} \label{eq:Mt}
		R_j^t f = M^t(R_j f), \qquad f\in L^2.
	\end{equation}
	Moreover the maximal operator $M^*$ is bounded on all $L^p$ spaces, $1<p<\infty$, and the optimal constant $C_p$ in the inequality $\norm{R_j^* f}_p \leqslant C_p\norm{R_j f}_p$ equals $\norm{M^*}_{p \to p}$.
\end{corollary}
\begin{proof}
Recalling \eqref{eq:ftRiesz} the decomposition \eqref{eq:Mt} follows immediately from \thref{lem1}.
	
When studying the $L^p$ boundedness of $M^*$ by \thref{lem:dens} it suffices to consider $M^* g$ with $g\in R_j(L^p) \cap \mathcal S.$ Note that for such $g$ the function $M^*g$ is measurable. Clearly, \eqref{eq:Mt} implies $C_p \leqslant \norm{M^*}_{p \to p}$. Applying  \eqref{eq:mat_ver} we see that $M^*$ is bounded on $R_j(L^p)\cap \mathcal S,$ which is a dense subset of $L^p$ by \thref{lem:dens}.  Thus, using again \eqref{eq:Mt} we see that $\norm{M^*}_{p \to p} \leqslant C_p.$ This completes the proof of the corollary.
	
\end{proof}

\begin{remark}
	\label{rem:factker}
	Using the argumentation in \cite[Section 2]{mateu_verdera} it is possible to obtain the factorization in Corollary \ref{cor:fact} on the kernel level. Namely, for $t>0,$ we have $	M^t f= h_t * f ,$ where $h_t(x)=-t^{-d}h(x/t)$ with $h$ being the function
	$
	h=\sum_{j=1}^d R_j(K_j^1),
	$
	see \cite[eq.\ (4), p.\ 959]{mateu_verdera}.  However, working at the kernel level makes it harder to obtain dimension-free estimates.
	
\end{remark}
\begin{remark}
	Corollary \ref{cor:fact} reduces the question of the $L^p$ control  of $R_j^*$ by $R_j$ to the boundedness of the maximal operator $M^*$ with equality of norms. As we noted in the proof of the corollary inequality \eqref{eq:mat_ver} implies that $M^*$ is indeed bounded on each $L^p,$ $1<p<\infty.$ However, \eqref{eq:mat_ver} only gives an estimate of an order $C_p \cdot 4^d$ for its $L^p$ norm, and the question of a dimension-free estimate on $L^p$, $p\neq 2$, for $M^*$ remains open. 
\end{remark}
\begin{remark}
	A necessary condition for a dimension-free boundedness of $M^*$ on $L^p$ is a dimension-free boundedness on $L^p$ of the single operator $M^1.$ This is also not clear to us. What makes matters more complicated is the fact that the kernel function $h_1$ is not non-negative. Thus, in certain aspects our context furnishes difficulties that were not present in the case of dimension-free estimates for Hardy--Littlewood maximal functions over convex sets studied in \cite{bou1,B2,B3,Car1,Mul1,SteinMax,SteinStro}. Indeed Hardy--Littlewood maximal operators are suprema over averaging operators, and each of these averaging operators is a positivity-preserving contraction on all $L^p$ spaces with $1\leqslant p\leqslant \infty.$ 
\end{remark}

By Corollary \ref{cor:fact}, \thref{thm} is equivalent to the following result.
\begin{theorem} \thlabel{thm'}
	Let $M^*$ be defined as in \eqref{eq:Mtmaxdef}. Then for every $f \in L^2(\R^d)$ we have
	\[
		\norm{M^* f}_2 \leqslant 2\cdot 10^8 \norm{f}_2.
	\]
\end{theorem}
Till the end of this section we focus on proving \thref{thm'} and on the operator $M^*$.

First we prove estimates for the multiplier $m$. We start with small arguments.

\begin{lemma} \thlabel{lem2}
	For $0\leqslant x \leqslant \sqrt{d}$ we have
	\[
		\abs{m(x) - 1} \leqslant 20 \frac{x}{\sqrt{d}}.
	\]
\end{lemma}

\begin{proof}
	By \cite[10.22.43]{nist} we know that $m(0) = 1$, so
	\begin{equation} \label{eq:m1}
		m(x) - 1 = m(x) - m(0) = -\frac{2^{\frac{d}{2}} \Gamma\left( \frac{d+1}{2} \right)}{\sqrt{\pi}} \int_0^{2\pi x} r^{-\frac{d}{2}} J_{\frac{d}{2}}(r) dr.
	\end{equation}
	Now \cite[B.6]{grafakos} gives
	\[
		J_\nu(x) = \frac{x^\nu}{2^\nu \Gamma(\nu+1)} + S_\nu(x)
	\]
	with $S_\nu$ satisfying
	\[
		\abs{S_\nu(x)} \leqslant \frac{2^{-\nu} x^{\nu+1}}{(\nu+1) \Gamma(\nu+\frac{1}{2}) \sqrt{\pi}}.
	\]
	Hence, using \eqref{eq:gautschi} we estimate \eqref{eq:m1} as follows (recall that $\frac{x}{\sqrt{d}} \leqslant 1$)
	\begin{align*}
		\abs{m(x) - 1} &\leqslant \frac{2\pi x \Gamma\left( \frac{d+1}{2} \right)}{\Gamma\left( \frac{d}{2}+1 \right) \sqrt{\pi}} + \frac{1}{\left(\frac{d}{2}+1 \right) \pi} \int_0^{2\pi x} r dr \\
		&\leqslant \frac{2\sqrt{2 \pi} x}{\sqrt{d}} + \frac{4 \pi x^2}{d} \leqslant 20\frac{x}{\sqrt{d}}.
	\end{align*}
\end{proof}

Our estimate for $m(x)$ when $x$ is large will be based on an inequality for the Bessel function $J_{\nu}.$ This is essentially a restatement of \cite[Lemma 4.1]{miszwr1}. We present the proof in order to keep track of the constants.

\begin{lemma} \thlabel{lemJ}
For each $t \geqslant 0$ and $\nu \geqslant 0$ we have
	\begin{equation*} 
		|J_{\nu}(t)| \leqslant \frac{2100 \, t^{\nu}}{2^{\nu}\Gamma\left( \nu + \frac{1}{2} \right) \sqrt{\nu \pi}} \left( e^{-\frac{t}{\sqrt{\nu}}} + e^{-\frac{\nu}{5}} \right).
		\end{equation*}
\end{lemma}

\begin{proof}
Define for $t\geqslant$ and $\nu\geqslant 0$ 
	\begin{equation} \label{eq:bessM}
		M(t) := \sqrt{\nu} \int_{-1}^1 e^{its\sqrt{\nu}} \left(1 - s^2 \right)^{\nu - \frac{1}{2}} ds= \int_{-\sqrt{\nu}}^{\sqrt{\nu}} e^{its} \left( 1 - \frac{s^2}{\nu} \right)^{\nu - \frac{1}{2}} ds. 
	\end{equation}
	Then, using the definition of the Bessel function \eqref{eq:besdef} we see that
	\[
J_\nu(t)= \frac{t^{\nu}}{2^{\nu}\Gamma\left( \nu + \frac{1}{2} \right) \sqrt{\nu \pi}}M\left(\frac{t}{\sqrt{\nu}}\right).
\]
Therefore in order to prove the lemma it suffices to show that 
	\begin{equation}
		\label{eq:lemJaim}
		\abs{M(t)} \leqslant 2100 \left( e^{-t} + e^{-\frac{\nu}{5}} \right), \quad t \geqslant 0,
	\end{equation}
and till the end of the proof we focus on justifying \eqref{eq:lemJaim}.

	We begin by splitting the second integral in \eqref{eq:bessM} into two parts
	\begin{equation} \label{eq:bess2}
		\abs{M(t)} \leqslant \abs{\int_{\frac{\sqrt{\nu}}{2} \leqslant \abs{s} \leqslant \sqrt{\nu}} e^{its} \left( 1 - \frac{s^2}{\nu} \right)^{\nu - \frac{1}{2}} ds} + \abs{\int_{\abs{s} \leqslant \frac{\sqrt{\nu}}{2}} e^{its} \left( 1 - \frac{s^2}{\nu} \right)^{\nu - \frac{1}{2}} ds}.
	\end{equation}
	Then we observe that
	\[
		\abs{\int_{\frac{\sqrt{\nu}}{2} \leqslant \abs{s} \leqslant \sqrt{\nu}} e^{its} \left( 1 - \frac{s^2}{\nu} \right)^{\nu - \frac{1}{2}} ds} \leqslant 2\sqrt{\nu} \left( \frac{3}{4} \right)^{\nu-\frac{1}{2}} \leqslant 6e^{-\frac{\nu}{4}}
	\] 
	since $1-\frac{s^2}{\nu} \leqslant \frac{3}{4}$ for $\abs{s} \geqslant \frac{\sqrt{\nu}}{2}$. This means that we can move on to estimating the second integral in \eqref{eq:bess2}. To do this we will change the contour of integration. However this will work only if $\nu > \frac{4}{3}$, so we take care of $\nu \leqslant \frac{4}{3}$ first. In this case we use \eqref{eq:gautschi} to write
	\begin{align*}
		\abs{\int_{\abs{s} \leqslant \frac{\sqrt{\nu}}{2}} e^{its} \left( 1 - \frac{s^2}{\nu} \right)^{\nu - \frac{1}{2}} ds} \leqslant \sqrt{\nu}\int_{-1}^1 \left( 1 - s^2 \right)^{\nu - \frac{1}{2}} ds 
		= \sqrt{\nu} \frac{\sqrt{\pi} \Gamma \left( \nu + \frac{1}{2} \right)}{\Gamma \left( \nu + 1 \right)} \leqslant \sqrt{\pi} \leqslant 3e^{-\frac{\nu}{5}}.
	\end{align*}

	Now assume that $\nu > \frac{4}{3}$ and let $\gamma=\gamma_{0}\cup\gamma_{1}\cup\gamma_2\cup\gamma_3$ be the rectangle with the parametrization
	\begin{align*}
		\gamma_0(s)& := s						&&\text{for} \quad s \in \left[-\tfrac{\sqrt{\nu}}{2}, \tfrac{\sqrt{\nu}}{2} \right],\\
		\gamma_1(s)& := is+\frac{\sqrt{\nu}}{2} 	&&\text{for} \quad s \in [0,1],\\
		\gamma_2(s)& := -s+i					&&\text{for} \quad s \in \left[-\tfrac{\sqrt{\nu}}{2}, \tfrac{\sqrt{\nu}}{2} \right],\\
		\gamma_3(s)& := i(1-s)-\frac{\sqrt{\nu}}{2}	&&\text{for} \quad s \in [0,1].
	\end{align*}
	The function $z \mapsto e^{itz} \left( 1 - \frac{z^2}{\nu} \right)^{\nu - \frac{1}{2}}$ is holomorphic in $\{z \in \C: \abs{z} < \sqrt{\nu}\}$ so for $\nu > \frac{4}{3}$ the Cauchy integral theorem gives
	\begin{align*}
		\abs{\int_{\abs{s} \leqslant \frac{\sqrt{\nu}}{2}} e^{its} \left( 1 - \frac{s^2}{\nu} \right)^{\nu - \frac{1}{2}} ds} &\leqslant \sum_{j \in \{1,3\}} \abs{\int_0^1 e^{it\gamma_j(s)} \left(1 - \frac{\gamma_j(s)^2}{\nu} \right)^{\nu - \frac{1}{2}} ds} \\
		&+ \abs{\int_{\abs{s} \leqslant \frac{\sqrt{\nu}}{2}} e^{it(i-s)} \left( 1 - \frac{(s-i)^2}{\nu} \right)^{\nu - \frac{1}{2}} ds}.
	\end{align*}
	The first term can be estimated as
	\begin{align*}
		\sum_{j \in \{1,3\}} \abs{\int_0^1 e^{it\gamma_j(s)} \left(1 - \frac{\gamma_j(s)^2}{\nu} \right)^{\nu - \frac{1}{2}} ds} \leqslant 2\left( \frac{3}{4} + \frac{1}{\nu} + \frac{1}{\sqrt{\nu}} \right)^{\nu - \frac{1}{2}} \leqslant 600 e^{-\frac{\nu}{5}}.
	\end{align*}
	Since $e^{it(i-s)} = e^{-t}e^{-its}$, now it suffices to show that
	\[
		\abs{\int_{\abs{s} \leqslant \frac{\sqrt{\nu}}{2}} e^{-its} \left( 1 - \frac{(s-i)^2}{\nu} \right)^{\nu - \frac{1}{2}} ds} \leqslant 2100.
	\]
	Recall that $\nu > \frac{4}{3}$ and observe that
	\[
		\abs{1 - \frac{(s-i)^2}{\nu}} \leqslant 1 - \frac{s^2-1}{\nu} + \frac{2\abs{s}}{\nu} \leqslant
		\begin{cases}
			1 + \frac{6}{\nu}, &\text{if } \abs{s} \leqslant \frac{5}{2} \\
			1 - \frac{s^2}{25 \nu} &\text{if } \frac{5}{2} \leqslant \abs{s} \leqslant \frac{\sqrt{\nu}}{2}
		\end{cases}
	\]
	and thus
	\begin{align*}
		\abs{\int_{\abs{s} \leqslant \frac{\sqrt{\nu}}{2}} e^{-its} \left( 1 - \frac{(s-i)^2}{\nu} \right)^{\nu - \frac{1}{2}} ds} &\leqslant 5\left( 1 + \frac{6}{\nu} \right)^{\nu - \frac{1}{2}} + \int_{\abs{s} \leqslant \frac{\sqrt{\nu}}{2}} \left( 1 - \frac{s^2}{25 \nu} \right)^{\nu - \frac{1}{2}} ds \\
		&\leqslant 5e^6 + 5\sqrt{\nu} \int_{-1}^1 \left(1 - s^2 \right)^{\nu - \frac{1}{2}} ds \\
		&= 5e^6 + 5\sqrt{\nu} \frac{\sqrt{\pi} \Gamma \left( \nu + \frac{1}{2} \right)}{\Gamma \left( \nu + 1 \right)} \\
		&\leqslant 5e^6 + 5\sqrt{\pi} \leqslant 2100.
	\end{align*}

This completes the proof of \eqref{eq:lemJaim} and thus also the proof of \thref{lemJ}.
	
\end{proof}

Applying \thref{lemJ} we now justify an estimate of $m$ for large arguments.

\begin{lemma} \thlabel{lem3}
	For $x \geqslant \sqrt{d}$ we have
	\[
		\abs{m(x)} \leqslant 6\cdot 10^4 \frac{\sqrt{d}}{x}.
	\]
\end{lemma}

\begin{proof}
	We consider two cases. First we take $x \geqslant d$. Recalling that $\abs{J_\nu(x)} \leqslant 1$ and $d \geqslant 4$ we can estimate the integral in \eqref{eq:m} by
	\[
		\abs{\int_{2\pi x}^\infty r^{-\frac{d}{2}} J_{\frac{d}{2}}(r) dr} \leqslant \int_{2\pi x}^\infty r^{-\frac{d}{2}} dr \leqslant 2 \frac{\left( 2\pi x \right)^{1-\frac{d}{2}}}{d-2} \leqslant 4 \frac{\left( 2\pi d \right)^{1-\frac{d}{2}}}{x}.
	\]
	Including the constant in \eqref{eq:m} and using \eqref{eq:stirling} for $d\geqslant 4$ gives
	\begin{equation} \label{eq1}
		\abs{m(x)} \leqslant 8\frac{\left( \pi d \right)^{1-\frac{d}{2}}}{x} \Gamma\left( \frac{d+1}{2} \right) \leqslant \frac{8}{x} \sqrt{2\pi} \pi d (2\pi e)^{-\frac{d}{2}} e^{\frac{1}{6(d+1)}} \leqslant \frac{1}{x} \leqslant \frac{\sqrt{d}}{x}.
	\end{equation}
	
	The second case is when $\sqrt{d} \leqslant x \leqslant d$. Then the integral in \eqref{eq:m} can be split into two parts: from $2 \pi x $ to $2 \pi d$ and from $2 \pi d$ to infinity; namely
	\begin{equation*} 
		m(x) = \frac{2^{\frac{d}{2}} \Gamma\left( \frac{d+1}{2} \right)}{\sqrt{\pi}} \left( \int_{2\pi x}^{2\pi d} r^{-\frac{d}{2}} J_{\frac{d}{2}}(r) dr + \int_{2\pi d}^\infty r^{-\frac{d}{2}} J_{\frac{d}{2}}(r) dr \right) = I_1(x) + I_2.
	\end{equation*}
	The second integral can be estimated as in \eqref{eq1}
	\[
		\abs{I_2} = \abs{m(d)} \leqslant \frac{\sqrt{d}}{d} \leqslant \frac{\sqrt{d}}{x}.
	\]
	To handle $I_1$, we use \thref{lemJ} which gives
	\begin{align*}
		\abs{I_1(x)} &\leqslant \frac{2100\sqrt{2}}{\pi\sqrt{d}} \left(\int_{2 \pi x}^{2 \pi d} e^{-\frac{r\sqrt{2}}{\sqrt{d}}} dr + \int_{2 \pi x}^{2 \pi d} e^{-\frac{d}{10}} dr \right) \\
		&\leqslant \frac{2100}{\pi} e^{-\frac{2\sqrt{2}\pi x}{\sqrt{d}}} + 4200\sqrt{2d} e^{-\frac{d}{10}} \leqslant 6 \cdot 10^4\frac{\sqrt{d}}{x}.
	\end{align*}
	In the last inequality we used the fact that $e^{-x} \leqslant \frac{1}{x}$ for $x \geqslant 0$.
\end{proof}
We will also need an estimate of the derivative of $m$.
\begin{lemma} \thlabel{lem4}
	For all $x \geqslant 0$ we have
	\[
		\abs{xm'(x)} \leqslant 10^4.
	\]
\end{lemma}
\begin{proof}
	Differentiating \eqref{eq:m} gives
	\[
		m'(x) = -\frac{2\sqrt{\pi}\Gamma\left( \frac{d+1}{2} \right)}{(\pi x)^{\frac{d}{2}}} J_{\frac{d}{2}}(2\pi x).
	\]
	If $x \geqslant d$, then we can estimate $J_{\frac{d}{2}}$ by 1 and use \eqref{eq:stirling} to get
	\[
		\abs{xm'(x)} \leqslant \frac{2\Gamma\left( \frac{d+1}{2} \right)}{\pi^{\frac{d-1}{2}} d^{\frac{d}{2}-1 }} \leqslant 2\sqrt{2} \pi d (2\pi e)^{-\frac{d}{2}} e^{\frac{1}{6(d+1)}} \leqslant 3.
	\]
	Otherwise, when $x < d$, we use \thref{lemJ} which yields
	\begin{align*}
		\abs{xm'(x)} &\leqslant x\frac{2\sqrt{\pi}\Gamma\left( \frac{d+1}{2} \right)}{\left( \pi x \right)^{\frac{d}{2}}} \frac{2100 \sqrt{2} \, (2\pi x)^{\frac{d}{2}}}{2^{\frac{d}{2}}\Gamma\left( \frac{d+1}{2} \right) \sqrt{d \pi}} \left( e^{-\frac{2\sqrt{2} \pi x}{\sqrt{d}}} + e^{-\frac{d}{10}} \right) \\
		&= 4200\sqrt{2}\frac{x}{\sqrt{d}} \left( e^{-\frac{2\sqrt{2} \pi x}{\sqrt{d}}} + e^{-\frac{d}{10}} \right) \leqslant \frac{2100}{\pi} + \frac{4200\sqrt{10}}{\sqrt{e}} \leqslant 10^4.
	\end{align*}
\end{proof}

\section{Proof of Theorem \ref{thm'}} \label{sec:main}

Having established the technical results regarding the multiplier $m$, we move on to the proof of \thref{thm'}. We estimate $M^*$ as follows

\begin{equation} \label{eq:R*2}
	M^* f = \sup_{t > 0} \abs{M^t f} \leqslant \sup_{n \in \Z} \abs{M^{2^n} f} + \left( \sum_{n \in \Z} \sup_{t \in [2^n, 2^{n+1}]} \abs{M^t f - M^{2^n} f}^2 \right)^{1/2}.
\end{equation}

To bound the first part, we compare it with the maximal function of the Poisson semigroup
\begin{equation} \label{eq:R+}
	\sup_{n \in \Z} \abs{M^{2^n} f} \leqslant \sup_{n \in \Z} \abs{M^{2^n} f - P_{2^n} f} + \abs{P_* f}.
\end{equation}
Since by \eqref{eq:P*g} the norm of $P_*$ is bounded on $L^2$ by $4$, to estimate the first term in \eqref{eq:R*2} it is enough to take care of the first term in \eqref{eq:R+}.

\begin{theorem} \thlabel{th6}
For every $f\in L^2$ we have 
	\[
		\norm{\sup_{n \in \Z} \abs{M^{2^n} f}}_2 \leqslant 1.3\cdot 10^5 \norm{f}_2.
	\]
\end{theorem}

\begin{proof}
	As noted above in order to prove the theorem it is enough to show that
	\[
		\norm{\sup_{n \in \Z} \abs{M^{2^n} f - P_{2^n} f}}_2 \leqslant 1.2\cdot 10^5\norm{f}_2.
	\]
	Estimating the supremum by the sum and using Plancherel's theorem we arrive at
	\begin{align*}
		&\norm{\sup_{n \in \Z} \abs{M^{2^n} f - P_{2^n} f}}_2^2 
		\leqslant \sum_{n \in \Z} \norm{M^{2^n} f - P_{2^n} f}_2^2 = \sum_{n \in \Z} \norm{\widehat{M^{2^n} f} - \widehat{P_{2^n} f}}_2^2.
	\end{align*}
	Recall that the multiplier symbol associated with $M^t$ is $m(t \abs{\xi})$ by \eqref{eq:Mt} and the one of the Poisson semigroup $P_t$ is $e^{-t \frac{\abs{\xi}}{\sqrt{d}}}$ by its definition \eqref{eq:poisson}. Combining these facts leads to
	\begin{align} \label{eq5}
		\sum_{n \in \Z} \norm{\widehat{M^{2^n} f} - \widehat{P_{2^n} f}}_2^2 &= \sum_{n \in \Z} \norm{\left( m(2^n \abs{\xi}) - e^{-2^n \frac{\abs{\xi}}{\sqrt{d}}} \right) \widehat{f}}_2^2.
	\end{align}
	Now we need to estimate the expression inside the norm. We split the analysis into two cases in order to use \thref{lem2} and \thref{lem3}. First assume that $2^n\abs{\xi} \leqslant \sqrt{d}$. Then by \thref{lem2} and the fact that $1-e^{-x} \leqslant x$ we have
	\begin{equation} \label{eq3}
		\abs{m(2^n \abs{\xi}) - e^{-2^n \frac{\abs{\xi}}{\sqrt{d}}}} \leqslant \abs{m(2^n \abs{\xi}) - 1} + \abs{e^{-2^n \frac{\abs{\xi}}{\sqrt{d}}} - 1} \leqslant 21 \frac{2^n \abs{\xi}}{\sqrt{d}}.
	\end{equation}
	If, on the other hand, $2^n \abs{\xi} \geqslant \sqrt{d}$, then we use \thref{lem3} and the fact that $e^{-x} \leqslant \frac{1}{x}$ to get
	\begin{equation} \label{eq4}
		\abs{m(2^n \abs{\xi}) - e^{-2^n \frac{\abs{\xi}}{\sqrt{d}}}} \leqslant 6\cdot 10^4 \frac{\sqrt{d}}{2^n \abs{\xi}}.
	\end{equation}
	Combining \eqref{eq3} and \eqref{eq4} gives
	\[
		\abs{m(2^n \abs{\xi}) - e^{-2^n \frac{\abs{\xi}}{\sqrt{d}}}} \leqslant 6\cdot 10^4 \min \left( \frac{2^n \abs{\xi}}{\sqrt{d}}, \left( \frac{2^n \abs{\xi}}{\sqrt{d}} \right)^{-1} \right).
	\]
	Squaring and summing over $n \in \Z$ leads to
	\[
		\sum_{n \in \Z} \abs{m(2^n \abs{\xi}) - e^{-2^n \frac{\abs{\xi}}{\sqrt{d}}}}^2 \leqslant 4\cdot (6\cdot 10^4)^2,\qquad \xi\in \R^d.
	\]
	Plugging the inequality above into \eqref{eq5} finally gives
	\[
		\sum_{n \in \Z} \norm{\left( m(2^n \abs{\xi}) - e^{-2^n \frac{\abs{\xi}}{\sqrt{d}}} \right) \widehat{f}}_2^2 \leqslant (1.2\cdot 10^5)^2 \norm{\widehat{f}}_2^2  = (1.2\cdot 10^5)^2 \norm{f}_2^2.
	\]
	This completes the proof of \thref{th6}.
\end{proof}

Now we estimate the norm of the second term in \eqref{eq:R*2}.

\begin{theorem} \thlabel{th7}
For every $f\in \mathcal S$ we have 
	\[
		\norm{\left( \sum_{n \in \Z} \sup_{t \in [2^n, 2^{n+1}]} \abs{M^t f - M^{2^n} f}^2 \right)^{1/2}}_2 \leqslant 1.7\cdot 10^8 \norm{f}_2.
	\]	
\end{theorem}

\begin{proof}
	Since $f\in \mathcal S$, we see that for each $x\in \R^d$ the function $t\mapsto M^t f(x)$ is continuous. Hence, an application of the numerical inequality \eqref{eq:ineq} is legitimate. Using this inequality, the resolution of identity \eqref{eq:Sn} given by $S_n,$ and the triangle inequality on the space $L^2(\ell^2)$  we obtain
	\begin{align} \label{eq:th7}
		&\norm{\left( \sum_{n \in \Z} \sup_{t \in [2^n, 2^{n+1}]} \abs{M^t f - M^{2^n} f}^2 \right)^{1/2}}_2 \nonumber \\
		&\leqslant \sqrt{2} \sum_{l=0}^\infty \sum_{k \in \Z} \norm{\left( \sum_{n \in \Z} \sum_{m=0}^{2^l - 1} \abs{\left( M^{2^n + 2^{n-l} (m+1) } - M^{2^n + 2^{n-l} m } \right) S_{k+n} f}^2 \right)^{1/2}}_2.
	\end{align}
	Then we estimate the norm in the above expression in two ways.
	
	 Similarly to the previous proof, by Plancherel's theorem we have
	\begin{align*}
		&\norm{\left( \sum_{n \in \Z} \sum_{m=0}^{2^l - 1} \abs{\left( M^{2^n + 2^{n-l} (m+1) } - M^{2^n + 2^{n-l} m } \right) S_{k+n} f}^2 \right)^{1/2}}_2^2 \\
		&= \sum_{n \in \Z} \sum_{m=0}^{2^l - 1} \Big\lVert \left(m \left( (2^n + 2^{n-l} (m+1)) \abs{\xi} \right) - m \left( (2^n + 2^{n-l} m) \abs{\xi} \right) \right) \times \left( e^{-2^{n+k} \frac{\abs{\xi}}{\sqrt{d}}} - e^{-2^{n+k-1} \frac{\abs{\xi}}{\sqrt{d}}} \right) \widehat{f} \Big\rVert_2^2
	\end{align*}
	We estimate the first factor in the norm using \thref{lem2,lem3}
	\[
		\abs{m \left( (2^n + 2^{n-l} (m+1)) \abs{\xi} \right) - m \left( (2^n + 2^{n-l} m) \abs{\xi} \right)} \leqslant 3\cdot 10^5 \min\left( \frac{2^n \abs{\xi}}{\sqrt{d}}, \left( \frac{2^n \abs{\xi}}{\sqrt{d}} \right)^{-1} \right)
	\]
	and the second one by
	\[
		\abs{e^{-2^{n+k} \frac{\abs{\xi}}{\sqrt{d}}} - e^{-2^{n+k-1} \frac{\abs{\xi}}{\sqrt{d}}}} \leqslant 3\min\left( \frac{2^{n+k} \abs{\xi}}{\sqrt{d}}, \left( \frac{2^{n+k} \abs{\xi}}{\sqrt{d}} \right)^{-1} \right).
	\]
	The product of the right-hand sides of the two inequalities above can be further estimated by $10^6 \cdot 2^{-\abs{k}}$, which gives
	\begin{align} \label{eq:th71}
		&\norm{\left( \sum_{n \in \Z} \sum_{m=0}^{2^l - 1} \abs{\left( M^{2^n + 2^{n-l} (m+1) } - M^{2^n + 2^{n-l} m } \right) S_{k+n} f}^2 \right)^{1/2}}_2^2 \nonumber \\
		&\leqslant 10^{12} \cdot 2^{-\abs{k}} 2^l \int_{\R^d} \sum_{n \in \Z} \min\left( \frac{2^n \abs{\xi}}{\sqrt{d}}, \left( \frac{2^n \abs{\xi}}{\sqrt{d}} \right)^{-1} \right) \abs{\widehat{f}(\xi)}^2 d\xi \leqslant (2\cdot 10^6)^2 2^{-\abs{k}} 2^l \norm{f}_2^2.
	\end{align}
	
	For the second way of estimating \eqref{eq:th7} note that \thref{lem4} implies
	\begin{align*}
		&\abs{m \left( (2^n + 2^{n-l} (m+1)) \abs{\xi} \right) - m \left( (2^n + 2^{n-l} m) \abs{\xi} \right)} \\
		&\leqslant \int_{2^n + 2^{n-l} m}^{2^n + 2^{n-l} (m+1)} \abs{t \abs{\xi} m'(t\abs{\xi})} \frac{dt}{t} \leqslant 10^4 \cdot 2^{-l}.
	\end{align*}
	We use the above inequality, Plancherel's theorem and \eqref{eq:Snnorm} to continue \eqref{eq:th7} as follows
	\begin{align} \label{eq:th72}
		&\norm{\left( \sum_{n \in \Z} \sum_{m=0}^{2^l - 1} \abs{\left( M^{2^n + 2^{n-l} (m+1) } - M^{2^n + 2^{n-l} m } \right) S_{k+n} f}^2 \right)^{1/2}}_2^2 \nonumber \\
		& \leqslant 10^8 \cdot 2^{-l} \sum_{n \in \Z} \norm{S_{k+n} f}_2^2 \leqslant 10^8 \cdot 2^{-l} \norm{f}_2^2.
	\end{align}
	Putting \eqref{eq:th71} and \eqref{eq:th72} together we reach
	\begin{align*}
		\norm{\left( \sum_{n \in \Z} \sup_{t \in [2^n, 2^{n+1}]} \abs{M^t f - M^{2^n} f}^2 \right)^{1/2}}_2 &\leqslant 2\cdot 10^6 \sqrt{2} \sum_{l=0}^\infty \sum_{k \in \Z} 2^{-\frac{l}{2}} \min \left(1, 2^{l-\frac{\abs{k}}{2}} \right) \norm{f}_2 \\
		&\leqslant 1.7\cdot 10^8\norm{f}_2.
	\end{align*}

The proof of \thref{th7} is completed.
\end{proof}
In the light of \eqref{eq:R*2}, \thref{th6} and \thref{th7}, the proof of  \thref{thm} is concluded.

\section{Proof of Theorem \ref{thm2}}

	Throughout the proof we fix $p\in(1,\infty).$ The reasoning is a minor variation on \cite{rubio}. We use the method of rotations which lets us express the Riesz transform in terms of the Hilbert transform. For basic facts regarding Hilbert transforms used in the proof see \cite[Sections 5.1 and 5.2.3]{grafakos}. For a unit vector $\theta \in \R^d$ and $\varepsilon > 0$ define 
	\[
		H_\theta^\varepsilon f(x) = \frac{1}{\pi} \int_{\abs{t} > \varepsilon} f(x-t\theta) \frac{dt}{t} \quad \text{and} \quad H_\theta^* f(x) = \sup_{\varepsilon > 0} \abs{H_\theta^\varepsilon f(x)};
	\]
	here we take $x\in \R^d$ and $f\in \mathcal S.$ The latter operator is the {directional maximal Hilbert transform}.
	For a Schwartz function $g\colon \R \to \C$ we also let 
		\[
	H^\varepsilon g(x) = \frac{1}{\pi} \int_{\abs{t} > \varepsilon} g(x-t) \frac{dt}{t} \quad \text{and} \quad H^* g(x) = \sup_{\varepsilon > 0} \abs{H^\varepsilon g(x)}
	\]
	be the truncated Hilbert transform on $\R$ and the maximal Hilbert transform on $\R,$ respectively. Since $H^*$ is bounded on $L^p(\R)$ for $p \in (1, \infty)$, this is also the case for $H_\theta^*$ on $L^p(\R^d)$. Moreover, the $L^p(\R^d)$ norm of $H_\theta^*$ equals the $L^p(\R)$ norm of $H^*$. In particular $\|H_{\theta}^*\|_{p\to p}$ depends only on $p\in(1,\infty)$ and not on $\theta$ nor on the dimension $d.$ For details on these statements see the proof of \cite[Theorem 5.2.7]{grafakos}.	  
	
	Using \cite[5.2.20]{grafakos} we may write
	\[
		R_j^tf(x) = \frac{c_d \pi}{2} \int_{S^{d-1}} \theta_j H_\theta^t f(x) d\theta,
	\]
	where $c_d$ is defined in \eqref{eq:cd} and $S^{d-1}$ denotes the unit sphere in $\R^d$.
	Now observe that for every $x \in \R^d$ there exists $\lambda(x) = (\lambda_1, \dots, \lambda_d) \in S^{d-1}$ such that
	\[
		\abs{R^t f(x)} = \sum_{j=1}^d \lambda_j(x) R_j^t f(x) = \frac{c_d \pi}{2} \int_{S^{d-1}} \innprod{\lambda(x)}{\theta} H_\theta^t f(x) d\theta.
	\]
	Taking the supremum of both sides of the above equality we obtain
	\[
		\sup_{t > 0} \abs{R^t f(x)} \leqslant \frac{c_d \pi}{2} \int_{S^{d-1}} \abs{\innprod{\lambda(x)}{\theta}} H_\theta^* f(x) d\theta.
	\]
	Then we use H\"{o}lder's inequality with $\frac{1}{p}+\frac{1}{q} = 1$ and the fact that the spherical measure is invariant under rotations, which results in
	\begin{align*} 
		\sup_{t > 0} \abs{R^t f(x)} &\leqslant \frac{c_d \pi}{2} \left( \int_{S^{d-1}} \abs{\innprod{\lambda(x)}{\theta}}^q d\theta \right)^{1/q} \left( \int_{S^{d-1}} H_\theta^* f(x)^p d\theta \right)^{1/p} \nonumber \\
		&= \frac{c_d \pi}{2} \left( \int_{S^{d-1}} \abs{\theta_1}^q d\theta \right)^{1/q} \left( \int_{S^{d-1}} H_\theta^* f(x)^p d\theta \right)^{1/p}.
	\end{align*}
	The first integral can be estimated with the help of \cite[(1)]{rubio} which says that
	\[
		\int_{S^{d-1}} \abs{\theta_1}^q d\theta = S_{d-1} \frac{\Gamma \left( \frac{d}{2} \right) \Gamma \left( \frac{q+1}{2} \right)}{\sqrt{\pi} \Gamma \left( \frac{d+q}{2} \right)},
	\]
	where $S_{d-1}: = \frac{2 \pi^{d/2}}{\Gamma \left( \frac{d}{2} \right)}$ is the surface area of $S^{d-1}$. This leads to the estimate 
		\begin{align} \label{eq:vecR}
		\norm{\sup_{t > 0} \abs{R^t f(x)}}_p &\lesssim c_d S_{d-1}^{1/q} \left( \frac{\Gamma \left( \frac{d}{2} \right) \Gamma \left( \frac{q+1}{2} \right)}{\sqrt{\pi} \Gamma \left( \frac{d+q}{2} \right)} \right)^{1/q} \left( \int_{\R^d} \int_{S^{d-1}} H_\theta^* f(x)^p d\theta dx \right)^{1/p} \nonumber \\
		&= c_d S_{d-1}^{1/q} \left( \frac{\Gamma \left( \frac{d}{2} \right) \Gamma \left( \frac{q+1}{2} \right)}{\sqrt{\pi} \Gamma \left( \frac{d+q}{2} \right)} \right)^{1/q} \left( \int_{S^{d-1}} \norm{H_\theta^* f}_p^p d\theta \right)^{1/p} \nonumber \\
		&\leqslant c_d S_{d-1} \left( \frac{\Gamma \left( \frac{d}{2} \right) \Gamma \left( \frac{q+1}{2} \right)}{\sqrt{\pi} \Gamma \left( \frac{d+q}{2} \right)} \right)^{1/q} \norm{H^*}_{L^p(\R)\to L^p(\R)}\, \norm{f}_p.
	\end{align}
The symbol $\lesssim $ in the inequality above denotes an estimate up to a multiplicative factor independent of the dimension. Now we use Stirling's formula (more precisely inequality \eqref{eq:stirling}) to estimate the gamma functions above as follows
	\begin{align*}
		\frac{\Gamma \left( \frac{d}{2} \right) \Gamma \left( \frac{q+1}{2} \right)}{\Gamma \left( \frac{d+q}{2} \right)} &\leqslant \sqrt{2\pi} \frac{\left(\frac{d}{2}\right)^\frac{d-1}{2} e^{-\frac{d}{2} + \frac{1}{6d}} \left( \frac{q+1}{2} \right)^\frac{q}{2} e^{-\frac{q+1}{2} + \frac{1}{6(q+1)}}}{\left( \frac{d+q}{2} \right)^{\frac{d+q-1}{2}} e^{-\frac{d+q}{2}}} \\
		&\leqslant \sqrt{2\pi} \left(\frac{d}{2}\right)^\frac{d-1}{2} \left( \frac{q+1}{2} \right)^\frac{q}{2} \left( \frac{d}{2} \right)^{-\frac{d+q-1}{2}} e^{-\frac{1}{2} + \frac{1}{6d} + \frac{1}{6(q+1)}} \\
		&\leqslant \sqrt{2\pi} \left(\frac{d}{2}\right)^{-\frac{q}{2}} \left( \frac{q+1}{2} \right)^\frac{q}{2}.
	\end{align*}
	Using the estimate $c_d S_{d-1} \leqslant \sqrt{\frac{2d}{\pi}}$  (see \cite[(2)]{rubio}),  the above calculation, and the $L^p(\R)$ boundedness of $H^*$  we continue \eqref{eq:vecR} and obtain
	\begin{align*}
		\norm{\sup_{t > 0} \abs{R^t f(x)}}_p \lesssim C_p'\, 2^{\frac{1}{2q}} \sqrt{\frac{2d}{\pi}} \sqrt{\frac{2}{d}} \sqrt{\frac{q+1}{2}} \norm{f}_p  \leqslant C_p \norm{f}_p,
	\end{align*}
where $C_p'$ and $C_p$ are constants that depend only on $p\in(1,\infty).$
Thus, \thref{thm2} is justified.

\end{document}